\documentclass[psamsfonts]{amsart}
\usepackage{amssymb,amsfonts,lacromay}
\usepackage[all,arc]{xy}
\usepackage{enumerate}
\usepackage{color}

\usepackage{mathrsfs}

\usepackage{picture}
\usepackage{graphicx}

\usepackage{pstricks,color}
\usepackage{pst-plot}

\newrgbcolor{tauzerocolor}{0.5 0.5 0.5}

\newrgbcolor{twoextncolor}{0.7 0.7 0}
\newrgbcolor{etaextncolor}{0.5 0 1}
\newrgbcolor{nuextncolor}{0.6 0.3 0}

\newgray{gridline}{0.8}

\newcommand{\cirrad}{0.06}

\newpsobject{twoextn}{psline}{linecolor=twoextncolor}
\newpsobject{twoextncurve}{psbezier}{linecolor=twoextncolor}
\newpsobject{etaextn}{psline}{linecolor=etaextncolor}
\newpsobject{nuextn}{psline}{linecolor=nuextncolor}
\newpsobject{nuextncurve}{psbezier}{linecolor=nuextncolor}

\newtheorem{thm}{Theorem}[section]
\newtheorem{cor}[thm]{Corollary}

\newtheorem{lem}[thm]{Lemma}

\theoremstyle{definition}

\newtheorem{rem}[thm]{Remark}

\makeatletter
\let\c@equation\c@thm
\makeatother
\numberwithin{equation}{section}

\makeatletter
\ifx\SK@label\undefined\let\SK@label\label\fi
 \let\your@thm\@thm
 \def\@thm#1#2#3{\gdef\currthmtype{#3}\your@thm{#1}{#2}{#3}}
 \def\mylabel#1{{\let\your@currentlabel\@currentlabel\def\@currentlabel
  {\currthmtype~\your@currentlabel}
 \SK@label{#1@}}\label{#1}}
 
\makeatother

\title{The strong Kervaire invariant problem in dimension $62$}

\author{Zhouli Xu}
\address{Department of Mathematics, The University of Chicago, Chicago, IL 60637}
\email{xu@math.uchicago.edu}

\begin{document}

\maketitle

\begin{abstract}

Using a Toda bracket computation $\langle \theta_4, 2, \sigma^2\rangle$ due to Daniel C. Isaksen \cite{DI2}, we investigate the $45$-stem more thoroughly. We prove that $\theta_4^2=0$ using a $4$-fold Toda bracket. By \cite{BJM2}, this implies that $\theta_5$ exists and there exists a $\theta_5$ such that $2\theta_5=0$. Based on $\theta_4^2=0$, we simplify significantly the $9$-cell complex construction in \cite{BJM1} to a $4$-cell complex, which leads to another proof that $\theta_5$ exists.

\end{abstract}

\section{Introduction and main results}

The Kervaire invariant problem is one of the most interesting problems that relates geometric topology and stable homotopy theory. One way of formulating it, due to Browder \cite{Brd},  is in terms of the classical Adams spectral sequence (ASS) at the prime $2$:

$$\text{For each $n$, the element}~ h_n^2\in Ext^{2, 2^{n+1}-2} ~\text{survives in the ASS.}$$
\\
If $h_n^2$ survives, we denote the corresponding detecting elements in homotopy by $\theta_n\in\pi_{2^{n+1}-2}S^0$ and we say that $\theta_n$ exists. The strong Kervaire invariant problem for $n$ is the following.

$$\theta_n ~\text{exists, and there exists a}~ \theta_n ~\text{such that}~ 2 \theta_n =0.$$
\\
It is well-known that the first three Kervaire invariant elements $\theta_1, \theta_2$ and $\theta_3$ can be chosen to be $\eta^2, \nu^2$ and $\sigma^2$. And they all have order $2$. Mahowald and Tangora \cite{MT} showed that $\theta_4$ exists and $2 \theta_4=0$ by an ASS computation. In \cite{BJM1}, Barratt, Jones and Mahowald showed that $\theta_5$ exists by constructing a $9$-cell complex and using the Peterson-Stein formula. Recently, using equivariant homotopy technology, Hill, Hopkins and Ravenel \cite{HHR} in their marvelous paper showed that $\theta_n$ does not exist for all $n\geq 7$, which left the existence of $\theta_6$ as the only open case.\\

In \cite{BJM2}, Barratt, Jones and Mahowald gave the following inductive approach to the strong Kervaire invariant problem:

\begin{thm}
Suppose that there exists an element $\theta_n$ such that $2 \theta_n=0$ and $\theta_n^2=0$. Then there exists an element $\theta_{n+1}$ with $2 \theta_{n+1}=0$.
\end{thm}

In this paper, we prove the following:

\begin{thm}
$\theta_4^2=0$.
\end{thm}

Since $\theta_4$ is unique and $2 \theta_4=0$, we have the following corollary:

\begin{cor}
$\theta_5$ exists and there exists a $\theta_5$ such that $2 \theta_5=0$.
\end{cor}

\begin{rem}
In \cite{Mil}, R. J. Milgram claims to show that under the same condition as in Theorem 1.1, one has $\theta_{n+2}$ exists. If this were true, then we would have that $\theta_6$ exists. However, Milgram's argument fails because of a computational mistake \cite{Br3}.
\end{rem}

\begin{rem}
Note that if one can further prove that the same $\theta_5$ has the property $\theta_5^2=0$, then Theorem 1.1 will imply the open case $\theta_6$ exists and that there exists a $\theta_6$ such that $2 \theta_6=0$.\\
\end{rem}

For the case $\theta_5$, Lin \cite{Lin} shows that there exists a $\theta_5$ such that $2 \theta_5=0$ based on a computation of the Toda bracket $\langle \theta_4, 2, \sigma^2\rangle$. Based on the same Toda bracket but a different computational result, Kochman \cite{Ko} also shows that $\theta_4^2=0$ and hence that there exists a $\theta_5$ such that $2 \theta_5=0$. Recently, Isaksen \cite{DI2} computed this Toda bracket using more straightforward arguments. His result contradicts the results of both Lin and Kochman. For more details about where Lin and Kochman's arguments fails, see Remark 3.4. Our proof uses Isaksen's computation. Since Isaksen's computation of $\langle \theta_4, 2, \sigma^2\rangle$ gives a more complicated answer than the earlier claims, we must study several other Toda brackets to prove $\theta_4^2=0$.\\

Knowing $\theta_4^2=0$, we give a second proof of the existence of $\theta_5$. In \cite{BJM1}, Barratt, Jones and Mahowald constructed a $9$-cell complex $X'$, and maps $f':S^{62}\rightarrow X'$, $g':X'\rightarrow S^0$, such that the composite $g'\circ f' : S^{62}\rightarrow S^0$ realizes a $\theta_5$. We simplify this $9$-cell complex $X'$ into a $4$-cell complex $X$, and construct maps $f:S^{62}\rightarrow X$, $g:X\rightarrow S^0$ as indicated in the following cell diagram. We follow Barratt, Jones and Mahowald's notation of cell diagrams.

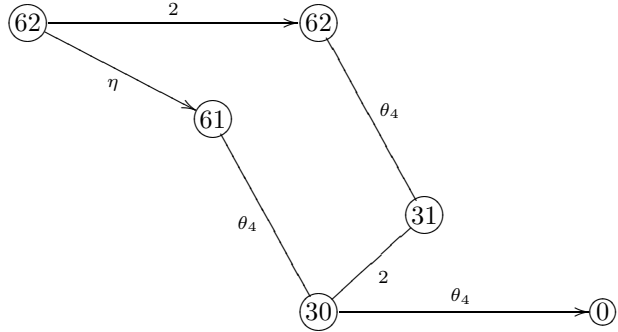
\begin{displaymath}
    \xymatrix{
   *+[o][F-]{62} \ar[rrr]^{2} \ar[drr]_{\eta} & & & *+[o][F-]{62} \ar@{-}[ddr]^{\theta_4} & & & \\
    & & *+[o][F-]{61} \ar@{-}[ddr]_{\theta_4} &  &  & & \\
  & &  &  & *+[o][F-]{31} \ar@{-}[dl]^{2} & & \\
   & &      &     *+[o][F-]{30} \ar[rrr]^{\theta_4} & & & *+[o][F-]{0} }
\end{displaymath}

Here each circle represents a cell. The number in each circle represents the dimension of that cell. The middle 4 cells represent the cell structure of $X$, where the three lines without arrow heads represent attaching maps of $X$. The map $g$ is an extension of $\theta_4$, and the map $f$ is a co-extension of $\eta\vee 2$. In other words, if we restrict the map $g$ on the bottom cell of $X$: $g|_{S^{30}}:S^{30}\rightarrow S^0$, we have $\theta_4$. If we pinch down the $31$-skeleton of $X$: $p:X\rightarrow S^{61}\vee S^{62}$, then the composite $p\circ f:S^{62}\rightarrow S^{61}\vee S^{62}$ is $\eta \vee 2$. For more details about cell diagrams, see \cite{BJM1}.

\begin{thm}
The composite of maps $g\circ f: S^{62}\rightarrow S^0$ realizes a $\theta_5$.
\end{thm}

\begin{proof}
We first show that we can form this cell diagram. For primary obstructions, we have $2 \theta_4=0$ and $\theta_4^2=0$. For secondary obstructions, we have $\eta\theta_4\in\langle 2, \theta_4, 2\rangle$ and $0\in\langle\theta_4, 2, \theta_4 \rangle$. The latter is shown in \cite{BJM1}. It is straightforward to check that the following two facts are true: for $i\leq 4$ the functional cohomology operations
$$Sq_g^{2^i}:H^0S^0\longrightarrow H^{2^i-1}X$$
are all zero, while $Sq_g^{32}:H^0S^0\rightarrow H^{31}X$ is nonzero; the functional cohomology operation $Sq_f^{32}$ is nonzero on $Sq_g^{32}H^0S^0=H^{31}X$. Note that all cohomology is understood to have mod 2 coefficients. As used in \cite{BJM1}, it follows from the Peterson-Stein formula (\cite{MoT},\cite{PS}) that the composite $g\circ f$ is detected by the secondary cohomology operation $\phi_{5,5}$. Therefore $g\circ f$ realizes a $\theta_5$.
\end{proof}

We present the proof of Theorem 1.2 in Section $2$. The proof uses several theorems and lemmas whose proofs we postpone. We include Isaksen's computation of $\langle \theta_4, 2, \sigma^2\rangle$ in Section $3$ for completeness. In Section $4$, we discuss two more Toda brackets in the $45$-stem, namely $\langle \theta_4, 2, \kappa\rangle$ and $\langle \theta_4, 2, \sigma^2+\kappa\rangle$. The proof of the main theorem depends on the computation of the latter bracket.  We give a modified $4$-fold Toda bracket for $\theta_4$ in Section $5$. We complete our proof of the main theorem by proving several lemmas in Section $6$.\\

\textbf{Acknowledgement}: The author would like to thank Dan Isaksen for discussing and sharing lots of his computations. The author would like to thank Dan Isaksen and Peter May for careful reading of several drafts of this paper. The author would also like to thank Bob Bruner for explaining the gap in Milgram's result. This paper is also just a tiny mark of our gratitude to Mark Mahowald for his tenacious exploration of the stable stems and his generosity towards us. The author would like to dedicate this paper to him, with special thanks for his inspiring weekly careful instruction and his guidance the year before his untimely death.

\section{The proof of the main theorem}

We will use the following Toda brackets to prove Theorem 1.2.

\begin{thm}
$\langle \theta_4, 2, \sigma^2+\kappa\rangle$ contains $0$ with indeterminacy $\{0, \rho_{15}\theta_4\}$.
\end{thm}

\begin{thm}
$\theta_4 = \langle 2, \sigma^2 + \kappa, 2 \sigma, \sigma \rangle$ with zero indeterminacy.
\end{thm}

\begin{lem}
$\sigma\pi_{53}=0$.
\end{lem}

\begin{lem}
$\langle\rho_{15}\theta_4, 2\sigma, \sigma\rangle=0$ with zero indeterminacy.
\end{lem}


We postpone the proof of Theorem 2.1 to Section 4, the proof of Theorem 2.2 to Section 5 and the proofs of Lemma 2.3 and 2.4 to Section 6. Now we present the proof of Theorem 1.2.

\begin{proof}
Following Theorems 2.1 and 2.2, we have

\begin{equation*}
\begin{split}
\theta_4^2 & = \theta_4 \langle 2, \sigma^2 + \kappa, 2 \sigma, \sigma \rangle \\
& \subseteq \langle\langle\theta_4, 2, \sigma^2 + \kappa\rangle, 2\sigma, \sigma\rangle \\
& = \text{the union of~~}  \langle 0, 2\sigma, \sigma\rangle \text{~and~} \langle \rho_{15}\theta_4, 2\sigma, \sigma \rangle
\end{split}
\end{equation*}

By Lemma 2.3 and Lemma 2.4 above, both brackets contain a single element zero. Therefore, we have that $\theta_4^2=0$.
\end{proof}

If $a$ is a surviving cycle in ASS, we use $\{a\}$ to denote the set of elements in the homotopy group that are detected by $a$. For elements in the $E_\infty$-page of the ASS, we include part of Isaksen's chart \cite{DI2}.

~\\

~\\


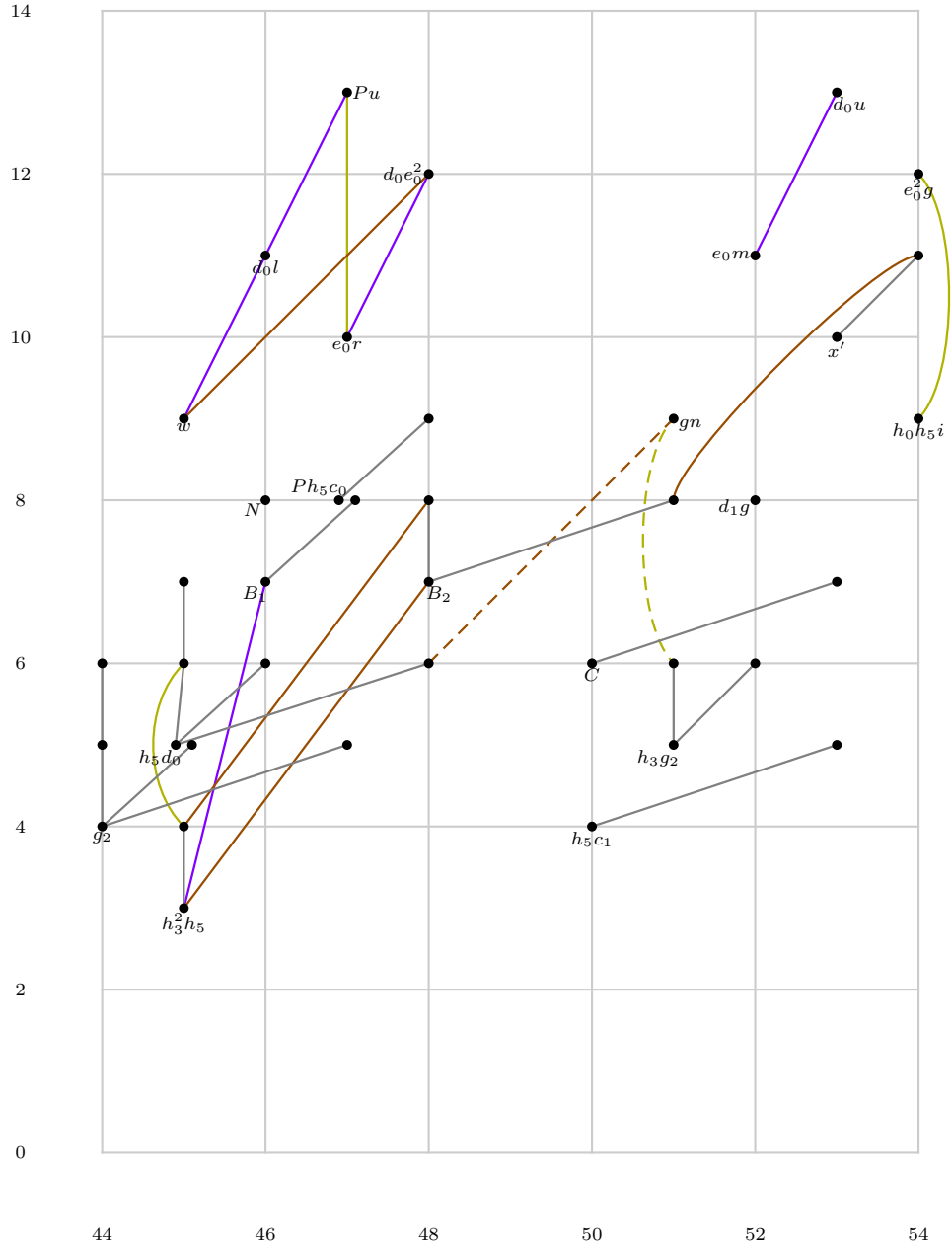
\psset{linewidth=0.3mm}

\psset{unit=1.1cm}
\begin{pspicture}(43,0)(54,14)

\psgrid[unit=2,gridcolor=gridline,subgriddiv=0,gridlabelcolor=white](22,0)(27,7)

\twoextncurve(45,4)(44.5,4.5)(44.5,5.5)(45,6)
\twoextn(47,10)(47,13)
\twoextncurve[linestyle=dashed](51,6)(50.5,6.5)(50.5,8.5)(51,9)
\twoextncurve(54,9)(54.5,9.5)(54.5,11.5)(54,12)

\etaextn(45,3)(46,7)
\etaextn(45,9)(46,11)
\etaextn(46,11)(47,13)
\etaextn(47,10)(48,12)
\etaextn(52,11)(53,13)

\nuextn(45,3)(48,7)
\nuextn(45,4)(48,8)
\nuextn(45,9)(48,12)
\nuextn[linestyle=dashed](48,6)(51,9)
\nuextncurve(51,8)(51,8.5)(53.5,11)(54,11)

\scriptsize

\rput(44,-1){44}
\rput(46,-1){46}
\rput(48,-1){48}
\rput(50,-1){50}
\rput(52,-1){52}
\rput(54,-1){54}

\rput(43,0){0}
\rput(43,2){2}
\rput(43,4){4}
\rput(43,6){6}
\rput(43,8){8}
\rput(43,10){10}
\rput(43,12){12}
\rput(43,14){14}

\psline[linecolor=tauzerocolor](44,4)(44,5)
\psline[linecolor=tauzerocolor](44,4)(45.10,5)
\psline[linecolor=tauzerocolor](44,4)(47,5)
\pscircle*(44,4){\cirrad}
\uput{\cirrad}[-90](44,4){$g_2$}
\psline[linecolor=tauzerocolor](44,5)(44,6)
\pscircle*(44,5){\cirrad}
\pscircle*(44,6){\cirrad}
\psline[linecolor=tauzerocolor](45,3)(45,4)
\pscircle*(45,3){\cirrad}
\uput{\cirrad}[-90](45,3){$h_3^2 h_5$}
\pscircle*(45,4){\cirrad}
\psline[linecolor=tauzerocolor](44.90,5)(45,6)
\psline[linecolor=tauzerocolor](44.90,5)(46,6)
\psline[linecolor=tauzerocolor](44.90,5)(48,6)
\pscircle*(44.90,5){\cirrad}
\uput{\cirrad}[240](44.90,5){$h_5 d_0$}
\pscircle*(45.10,5){\cirrad}
\psline[linecolor=tauzerocolor](45,6)(45,7)
\pscircle*(45,6){\cirrad}
\pscircle*(45,7){\cirrad}
\pscircle*(45,9){\cirrad}
\uput{\cirrad}[-90](45,9){$w$}
\pscircle*(46,6){\cirrad}
\psline[linecolor=tauzerocolor](46,7)(47.10,8)
\pscircle*(46,7){\cirrad}
\uput{\cirrad}[240](46,7){$B_1$}
\pscircle*(46,8){\cirrad}
\uput{\cirrad}[225](46,8){$N$}
\pscircle*(46,11){\cirrad}
\uput{\cirrad}[-90](46,11){$d_0 l$}
\pscircle*(47,5){\cirrad}
\psline[linecolor=tauzerocolor](46.90,8)(48,9)
\pscircle*(46.90,8){\cirrad}
\uput{\cirrad}[120](46.90,8){$P h_5 c_0$}
\pscircle*(47.10,8){\cirrad}
\pscircle*(47,10){\cirrad}
\uput{\cirrad}[-90](47,10){$e_0 r$}
\pscircle*(47,13){\cirrad}
\uput{\cirrad}[0](47,13){$P u$}
\pscircle*(48,6){\cirrad}
\psline[linecolor=tauzerocolor](48,7)(48,8)
\psline[linecolor=tauzerocolor](48,7)(51,8)
\pscircle*(48,7){\cirrad}
\uput{\cirrad}[-60](48,7){$B_2$}
\pscircle*(48,8){\cirrad}
\pscircle*(48,9){\cirrad}
\pscircle*(48,12){\cirrad}
\uput{\cirrad}[180](48,12){$d_0 e_0^2$}
\psline[linecolor=tauzerocolor](50,4)(53,5)
\pscircle*(50,4){\cirrad}
\uput{\cirrad}[-90](50,4){$h_5 c_1$}
\psline[linecolor=tauzerocolor](50,6)(53,7)
\pscircle*(50,6){\cirrad}
\uput{\cirrad}[-90](50,6){$C$}
\psline[linecolor=tauzerocolor](51,5)(51,6)
\psline[linecolor=tauzerocolor](51,5)(52,6)
\pscircle*(51,5){\cirrad}
\uput{\cirrad}[240](51,5){$h_3 g_2$}
\pscircle*(51,6){\cirrad}
\pscircle*(51,8){\cirrad}
\pscircle*(51,9){\cirrad}
\uput{\cirrad}[-30](51,9){$g n$}
\pscircle*(52,6){\cirrad}
\pscircle*(52,8){\cirrad}
\uput{\cirrad}[210](52,8){$d_1 g$}
\pscircle*(52,11){\cirrad}
\uput{\cirrad}[180](52,11){$e_0 m$}
\pscircle*(53,5){\cirrad}
\pscircle*(53,7){\cirrad}
\psline[linecolor=tauzerocolor](53,10)(54,11)
\pscircle*(53,10){\cirrad}
\uput{\cirrad}[-90](53,10){$x'$}
\pscircle*(53,13){\cirrad}
\uput{\cirrad}[-60](53,13){$d_0 u$}
\pscircle*(54,9){\cirrad}
\uput{\cirrad}[-90](54,9){$h_0 h_5 i$}
\pscircle*(54,11){\cirrad}
\pscircle*(54,12){\cirrad}
\uput{\cirrad}[-90](54,12){$e_0^2 g$}

\end{pspicture}

 \pagebreak


We do not include elements in filtration higher than 14. Those elements are detected by the $K(1)$-local sphere, and are not relevant to our proof. Here we use colored lines to denote nontrivial extensions. For example, the line between $Pu$ and $e_0r$ indicates that $2\{e_0r\}$ is nontrivial and is detected by $Pu$. The 2, $\eta$ and $\nu$-extensions are completely known in this range except for a possible 2-extension from $h_0h_3g_2$ to $gn$ and a possible $\nu$-extension from $h_2h_5d_0$ to $gn$. We use dashed lines to denote them. In fact, Isaksen \cite{DI1} showed that these two possible extensions either both occur or neither occur. But these extensions are irrelevant to our purpose.

\section{A Toda bracket $\langle \theta_4, 2, \sigma^2\rangle$}

The following theorem is due to Isaksen \cite{DI1}. For completeness, we include the proof.

\begin{thm}
$\langle \theta_4, 2, \sigma^2\rangle$ contains an element of order $2$ that can be detected by $h_0h_4^3$.
\end{thm}

\begin{rem}
Before presenting the proof, we mention that the indeterminacy of this Toda bracket is well-known. Namely, it is the set $\{0, \rho_{15}\theta_4\}$, where $\rho_{15}$ is the generator of $ImJ$ in $\pi_{15}$, and is detected by $h_0^3h_4$. Furthermore, $\rho_{15}\theta_4\neq 0$ is detected by $h_0^2h_5d_0$. This is shown by Tangora in \cite{Tan}.
\end{rem}

\begin{proof}
In the Adams $E_3$-page, we have $\langle h_4^2, h_0, h_3^2\rangle=h_4^2h_4+h_5h_3^2=0$ in the Adams filtration 3. Therefore, by the Moss Theorem \cite{Moss}, there is an element in $\langle \theta_4, 2, \sigma^2\rangle$ that is detected by some element of filtration at least $4$. Since the nontrivial element in the indeterminacy has filtration 7, any element in $\langle \theta_4, 2, \sigma^2\rangle$ has filtration at least 4. We have
$$2\langle \theta_4, 2, \sigma^2\rangle=\langle2, \theta_4, 2\rangle\sigma^2=\eta\theta_4\sigma^2=0.$$
Note that the indeterminacy of $\langle2, \theta_4, 2\rangle\sigma^2$ is $2\sigma^2\pi_{31}=0$. Therefore, any element in $\langle \theta_4, 2, \sigma^2\rangle$ has order $2$.

Now consider the product $\nu_4\theta_4$.
$$\nu_4\theta_4=\langle\sigma,\nu,\sigma\rangle\theta_4 \subseteq\langle\sigma,\nu,\sigma\theta_4\rangle\subseteq\langle\sigma,\nu,\{x\}\rangle.$$
Here, since $2\theta_4=0$, we can ignore the difference between $\nu_4$, which is by definition $\langle\nu,\sigma,2\sigma\rangle$, and $\langle\sigma,\nu,\sigma\rangle=7\nu_4$. In the Adams $E_2$-page, we have $h_2h_5d_0=\langle h_3,h_2,x\rangle$ with zero indeterminacy. In fact, this follows from
$$h_2\langle h_3,h_2,x\rangle=\langle h_2,h_3,h_2\rangle x=h_3^2x=h_2^2h_5d_0.$$ Therefore, $\nu_4\theta_4$ is contained in $\langle\sigma,\nu,\{x\}\rangle\subseteq\{h_2h_5d_0\}$.

On the other side, $\nu_4\theta_4$ is contained in $\theta_4\langle2, \sigma^2, \nu\rangle=\langle\theta_4,2,\sigma^2\rangle\nu$. For the indeterminacy, note that $\rho_{15}\theta_4\nu=0.$ Therefore, we actually have $$\nu_4\theta_4=\langle\theta_4,2,\sigma^2\rangle\nu.$$
Combining this with the fact that $\nu_4\theta_4$ is also contained in $\{h_2h_5d_0\}$, we deduce that there exists an element in $\langle \theta_4, 2, \sigma^2\rangle$ such that $\nu$ times it is detected by $h_2h_5d_0$, which has filtration $6$. Therefore, $\langle \theta_4, 2, \sigma^2\rangle$ contains an element with filtration at most $5$. Furthermore, it cannot be detected by $h_1g_2$, which has filtration $5$, since otherwise the $\nu$ multiple won't be detected by $h_2h_5d_0$. Therefore, the statement of the theorem is the only possibility left.
\end{proof}

\begin{rem}
Another way to describe the statement of this theorem is the following:
$$\langle \theta_4, 2, \sigma^2\rangle \text{~contains an order 2 element of the form~} 2\alpha+\beta,$$
where $\alpha$ is detected by $h_3^2h_5$ and $\beta$ is detected by $h_5d_0$. Note that the nontrivial 2-extension in the 45-stem means that there exist elements $\alpha$ and $\gamma$, which are detected by $h_3^2h_5$ and $h_5d_0$ respectively, such that $4\alpha =  2\gamma$. Since $\gamma$ has order 8, one can choose $\beta$ to be $-\gamma=7\gamma$, so that $2\alpha+\beta$ has order 2.
\end{rem}

\begin{rem}
In \cite{Lin}, Lin showed that this bracket contains 0. The step that rules out the element Isaksen got is invalid. In \cite{Ko}, Kochman showed that this bracket contains $\eta\{g_2\}$ or 0. His argument failed because essentially of the inconsistency of the $\nu-$extension on $\{h_2h_5d_0\}$ and the $\sigma-$extension on $\{h_0^2g_2\}$, which allowed him to eliminate the right element. The inconsistency is discussed in \cite{DI1}.
\end{rem}

\section{More about the $45$-stem}

We first consider the Toda bracket $\langle \theta_4, 2, \kappa\rangle$ in $\pi_{45}$.

\begin{lem}
$\langle \theta_4, 2, \kappa\rangle$ contains an element of order $2$ that can be detected by $h_0h_4^3$.
\end{lem}

\begin{proof}
The Adams differential $d_3(h_0h_4)=h_0d_0$ implies that in the Adams $E_4$-page, $\langle h_4^2, h_0, d_0\rangle = h_0h_4^3$ in the Adams filtration 4. Then by the Moss convergence theorem \cite{Moss}, there is an element in $\langle \theta_4, 2, \kappa\rangle$ that is detected by $h_0h_4^3$. From
$$2\langle \theta_4, 2, \kappa\rangle= \langle 2, \theta_4, 2\rangle\kappa=\eta\theta_4\kappa=0,$$
we know that any element in $\langle \theta_4, 2, \kappa\rangle$ has order $2$. The indeterminacy of $\langle 2, \theta_4, 2\rangle\kappa$ is $2\kappa\pi_{31}=0$. Here we also used that $\kappa\theta_4=0$, which is known for filtration reasons. In fact, since $d_0h_4^2=0$ in $Ext^6$, $\kappa\theta_4$ must be detected by an element of filtration at least 7. However, in the 44-stem of the $E_{\infty}$-page, there are no elements of filtration 7 or higher. Therefore $\langle \theta_4, 2, \kappa\rangle$ contains an element of order $2$ that can be detected by $h_0h_4^3$.
\end{proof}

\begin{rem}
The indeterminacy of this bracket is the same as that of $\langle\theta_4, 2, \sigma^2\rangle$, i.e., $\{0, \rho_{15}\theta_4\}$. In fact, $\pi_{31}$ is generated by $\eta\theta_4, \{n\}$ and $\rho_{31}$, where $\rho_{31}$ is the generator of $ImJ$ in $\pi_{31}$, and is detected by $h_0^{10}h_5$. Since $\kappa\theta_4=0$, $\eta\kappa\theta_4=0$. Again for filtration reasons, $\kappa\{n\}=0$ and $\kappa\rho_{31}=0$. Therefore $\kappa\pi_{31}=0$. This shows that the indeterminacy of $\langle \theta_4, 2, \kappa\rangle$ is $\{0, \rho_{15}\theta_4\}$.
\end{rem}

Although both $\langle \theta_4, 2, \kappa\rangle$ and $\langle \theta_4, 2, \sigma^2\rangle$ contain an element of order $2$ that can be detected by $h_0h_4^3$, we do not necessarily know if they have an element in common. The following theorem confirms that they do.

Now we restate Theorem 2.1.

\begin{thm}
$\langle \theta_4, 2, \sigma^2+\kappa\rangle$ contains $0$ with indeterminacy $\{0, \rho_{15}\theta_4\}$.
\end{thm}

We need the following lemma to prove the theorem.

\begin{lem}
$\sigma^2\pi_{33}=0$.
\end{lem}

\begin{proof}
We know that $\pi_{33}$ is generated by $\eta\eta_5$, $\nu\theta_4$, $\eta\{q\}$, $\eta^2\rho_{31}$ and $\{P^4h_1\}$. Since $\eta\sigma^2=0$ and $\nu\sigma^2=0$, we only need to show that $\{P^4h_1\}\sigma^2$=0. In fact, we have
$$\{P^4h_1\}\sigma^2=\eta\rho_{39}\sigma=0$$
for filtration reasons. Here $\rho_{39}$ is the generator of $ImJ$ in $\pi_{39}$, and is detected by $P^2h_0^2i$. Therefore, $\sigma^2\pi_{33}=0$.
\end{proof}

Now we present the proof of Theorem 4.3.

\begin{proof}
The indeterminacy is straightforward, as in Remark 4.2.

Since all elements in $\langle \theta_4, 2, \kappa\rangle$ and $\langle \theta_4, 2, \sigma^2\rangle$ have order $2$ and can be detected by $h_0h_4^3$ in the Adams filtration 4, elements in $\langle \theta_4, 2, \sigma^2+\kappa\rangle$ must be detected by elements of filtration at least $5$ and have order $2$. To prove the theorem, we need to rule out both $\{w\}$ and $\eta\{g_2\}$.

For $\{w\}$, by Lemma 4.4, we have that
$$\eta^2 \langle\theta_4, 2, \sigma^2\rangle  =  \langle\eta^2, \theta_4, 2\rangle \sigma^2 \in \pi_{33} \sigma^2 = 0.$$
Next we have that
$$\eta^2 \langle\theta_4, 2, \kappa\rangle  =  \theta_4 \langle 2, \kappa, \eta^2\rangle.$$

In the Adams $E_4$-page, we have that $\langle h_0, d_0, h_1^2\rangle= h_0h_4h_1^2 =0$ in the Adams filtration 4. Then the Moss Theorem tells us that $\langle 2,\kappa,\eta^2\rangle$ might contain a nontrivial element of higher filtration, namely a combination of $\nu\kappa, \eta^2\rho_{15}$ and $\{P^2h_1\}$. Note that we have that $\nu\kappa\theta_4=0$ and by Lemma 6.1 we have that $\eta^2\rho_{15}\theta_4=0$. To show that $\{P^2h_1\}\theta_4=0$, we first show that $\{Ph_1\}\theta_4=0$.\\

In fact, $\{Ph_1\}\theta_4\in \langle\eta,8\sigma, 2\rangle\theta_4 = \eta \langle 8\sigma, 2, \theta_4\rangle$, which contains $0$. This holds since $\eta \langle 8\sigma, 2, \theta_4\rangle$ intersects $\eta\{h_0^3h_3h_5\}$, which contains a single element zero. The indeterminacy is $\eta \pi_8 \theta_4 =0$. This gives that $\{Ph_1\}\theta_4=0$. Then we have

$$\{P^2h_1\}\theta_4 \in \theta_4 \langle \{Ph_1\}, 2, 8\sigma\rangle=\langle\theta_4, \{Ph_1\}, 2\rangle 8\sigma \subseteq \pi_{40} 8 \sigma=0.$$

Therefore, no matter what $\langle 2,\kappa,\eta^2\rangle$ equals, we always have that
$$\eta^2 \langle\theta_4, 2, \kappa\rangle  = \langle 2,\kappa, \eta^2\rangle\theta_4 \text{~contains~}0.$$

The indeterminacy of $\eta^2 \langle\theta_4, 2, \kappa\rangle$ is zero since $\eta^2\theta_4=0$ and $\eta^2\kappa=0$. Then
$$\eta^2 \langle\theta_4, 2, \kappa\rangle  =0.$$
Therefore,
$$\eta^2 \langle\theta_4, 2, \sigma^2+\kappa\rangle =0.$$
Then the fact that $\eta^2\{w\}\neq 0$ rules out $\{w\}$, since otherwise we would have that $\eta^2\langle\theta_4, 2, \sigma^2+\kappa\rangle = \eta^2 \{w\}\neq 0.$\\

For $\eta\{g_2\}$, first note that $\sigma\eta\{g_2\}\neq0$ is detected by $h_1h_3g_2$. We have that
$$\langle\theta_4, 2, \kappa\rangle \sigma = \theta_4 \langle 2, \kappa, \sigma\rangle \subseteq \theta_4\pi_{22} = 0.$$
In fact, $\pi_{22}$ is generated by $\nu\overline{\sigma}$ and $\eta^2\overline{\kappa}$. We have that $\eta^2\overline{\kappa}\theta_4=0$ and $\nu\overline{\sigma}\theta_4=0$ for filtration reasons. As a remark, we can actually prove that $\langle 2, \kappa, \sigma\rangle = \nu\overline{\sigma}$ by studying the cofiber of 2, but we don't need this fact here.

On the other side, as explained in Remark 3.3, $\langle\theta_4, 2, \sigma^2\rangle$ contains $2 \alpha + \beta.$ Therefore,
$$\langle\theta_4, 2, \sigma^2\rangle\sigma \text{~contains~} 2\alpha\sigma+\beta\sigma.$$

We have that $2\alpha\sigma \in 2 \pi_{52}=0$. In the Adams $E_3$-page, we compute directly that
$\langle h_0, h_4^2, d_0\rangle= h_5d_0$. Then Moss's Theorem shows that $\langle 2, \theta_4, \kappa\rangle$ contains an element that equals to $\beta$ plus possibly higher filtration terms. Note that $\sigma\{w\}=0$ by using tmf. In fact, if $\sigma\{w\}\neq 0$, the only possibility is that $\sigma\{w\}$ is detected by $\{e_0 m\}$. This implies that $\eta\sigma\{w\}=\kappa\{u\}$ because of the two nontrivial $\eta$-extensions. Since both $\eta\{w\}$ and $\kappa\{u\}$ are detected by tmf and $\sigma=0$ in $\pi_\ast tmf$, mapping this relation into tmf gives a contradiction. Besides, from tmf, we know that $\{d_0l\}$ detects $\kappa\{q\}$, then the contradiction also follows from $\kappa\sigma=0$. See \cite{Bau},\cite{He} for example. 

Then we have that
$$\beta\sigma \in \langle 2, \theta_4, \kappa\rangle \sigma = 2 \langle\theta_4, \kappa, \sigma\rangle \subseteq 2 \pi_{52}=0.$$
Therefore, $\langle\theta_4, 2, \sigma^2\rangle\sigma$ contains $2 \alpha\sigma + \beta\sigma=0$. Note that $\rho_{15}\theta_4\sigma\in\theta_4\pi_{22}=0$, the indeterminacy is hence zero. Then we have that
$$\langle\theta_4,2,\sigma^2\rangle\sigma=0.$$

Therefore,
$$\langle\theta_4,2,\sigma^2+\kappa\rangle\sigma=0. $$
Combined with the fact that $\eta\{g_2\}\sigma\neq 0$, this rules out $\eta\{g_2\}$.\\

This completes the proof.

\end{proof}

\begin{rem}
$\sigma^2+\kappa$ is another element in $\pi_{14}$ that deserves to be called $\theta_3$.
\end{rem}

\begin{rem}
We can actually show that the bracket $\langle 2, \theta_4, \eta^2\rangle$ contains $\eta\eta_5+\nu\theta_4$ with indeterminacy $\{0, \eta^2\rho_{31}\}$.
\end{rem}

\section{A modified $4$-fold Toda bracket for $\theta_4$}

We have the following well-known $4$-fold Toda brackets for $\theta_4$. See \cite{BMT},\cite{Ko},\cite{Ko2} for example.
\begin{equation*}
\begin{split}
\theta_4 & = \langle 2, \sigma^2, 2, \sigma^2 \rangle \\
& = \langle 2, \sigma^2, \sigma^2, 2 \rangle \\
& = \langle 2\sigma, \sigma, 2\sigma, \sigma \rangle \\
& = \langle 2, \sigma^2, 2 \sigma, \sigma \rangle
\end{split}
\end{equation*}

All of them have zero indeterminacy. This is partially discussed in \cite{BMT},\cite{Ko},\cite{Ko2}. For completeness, we include a proof here.

\begin{lem}
All four Toda brackets above have zero indeterminacy.
\end{lem}

\begin{proof}
In general, suppose a 4-fold Toda bracket $\langle \alpha_1, \alpha_2, \alpha_3, \alpha_4\rangle$ is defined, where $\alpha_i\in\pi_{n_i}$. Then its indeterminacy is contained in the union of three types of 3-fold Toda brackets:
$$\langle \alpha_1, \alpha_2, \pi_{n_3+n_4+1}\rangle, \langle \alpha_1, \pi_{n_2+n_3+1}, \alpha_4\rangle \text{~and~} \langle \pi_{n_1+n_2+1}, \alpha_3, \alpha_4\rangle. $$
In our case, the indeterminacy for all of them is contained in the union of the following eight brackets:
$$\langle\pi_{15},2, \sigma^2\rangle, \langle 2, \pi_{15}, \sigma^2\rangle, \langle 2, \sigma^2, \pi_{15}\rangle, \langle 2, \pi_{29}, 2\rangle, $$
$$\langle \pi_{15}, 2\sigma, \sigma\rangle, \langle 2\sigma, \pi_{15}, \sigma\rangle, \langle 2\sigma, \sigma, \pi_{15}\rangle, \langle 2, \pi_{22}, \sigma\rangle.$$
We will show that they are all zero. Note that $\pi_{30}\cong \mathbb{Z}/2$ and is generated by $\theta_4$, which is indecomposable. So for each of them, we only need to show that it does not contain $\theta_4$. They all follow for filtration reasons.

For $\langle\pi_{15},2, \sigma^2\rangle,\langle 2, \sigma^2, \pi_{15}\rangle,\langle \pi_{15}, 2\sigma, \sigma\rangle$ and $\langle 2\sigma, \sigma, \pi_{15}\rangle$, the corresponding Massey products are all well-defined on the Adams $E_3$-page. Since $\pi_{15}$ is generated by elements of filtration at least 4, the Massey products all take values in filtration at least 5. Therefore, by the Moss Theorem, all of them are all zero.

For $\langle 2, \pi_{15}, \sigma^2\rangle$ and $\langle 2\sigma, \pi_{15}, \sigma\rangle$, the corresponding Massey products are all well-defined on the Adams $E_2$-page. Since $\pi_{15}$ is generated by elements of filtration at least 4, the Massey products all take values in filtration at least 6. Therefore, by the Moss Theorem, all of them are all zero.

For $\langle 2, \pi_{22}, \sigma\rangle$, there are essentially two Toda brackets to check: $\langle 2, \nu\overline{\sigma}, \sigma\rangle$ and $\langle 2, \eta^2\overline{\kappa}, \sigma\rangle$, where $\nu\overline{\sigma}$ is detected by $h_2c_1$. Both brackets have zero indeterminacy. We have that
$$\langle 2, \nu\overline{\sigma}, \sigma\rangle=\langle 2, \overline{\sigma}, \nu\sigma\rangle=\langle 2, \overline{\sigma}, 0\rangle=0,$$
and that
$$\langle 2, \eta^2\overline{\kappa}, \sigma\rangle=\langle 2, \eta^2, \overline{\kappa}\sigma\rangle=\langle 2, \eta^2, 0\rangle=0.$$
Here we used the fact that $2 \overline{\sigma}=0$ and $\overline{\kappa}\sigma=0$.

At last, $\langle 2, \pi_{29}, 2\rangle=0$, since $\pi_{29}=0$. This completes the proof.
\end{proof}

Now we prove a modified $4$-fold Toda bracket based on the last one. Again, note that $\pi_{30}\cong \mathbb{Z}/2$ and is generated by $\theta_4$.

\begin{thm}
$\theta_4 = \langle 2, \sigma^2 + \kappa, 2 \sigma, \sigma \rangle$ with zero indeterminacy.
\end{thm}

\begin{proof}
We have $\langle\sigma^2 + \kappa, 2 \sigma, \sigma\rangle \subseteq \pi_{29} =0$. And
$$\langle 2, \sigma^2 + \kappa, 2 \sigma\rangle\supseteq\langle 2, \sigma^2 + \kappa, 2\rangle\sigma\ni\eta (\sigma^2 + \kappa) \sigma =0.$$
The indeterminacy of the bracket $\langle 2, \sigma^2 + \kappa, 2 \sigma\rangle$ is $2\pi_{22}+2\sigma\pi_{15}=0$, and we have $\langle 2, \sigma^2 + \kappa, 2 \sigma\rangle=0$. Therefore, this $4$-fold Toda bracket is strictly defined, and the indeterminacy is
$$\langle 2, \sigma^2 + \kappa, \pi_{15}\rangle + \langle 2, \pi_{22}, \sigma\rangle + \langle \pi_{15}, 2 \sigma, \sigma \rangle.$$
Note that $\langle 2, \sigma^2 + \kappa, \pi_{15}\rangle = 0$ for filtration reasons as in the proof of Lemma 5.1. The other two parts of the indeterminacy follow from the indeterminacy of $\langle 2, \sigma^2, 2 \sigma, \sigma \rangle$, which we know is zero. Then the theorem follows from the next lemma and the fact that $\theta_4 = \langle 2, \sigma^2, 2 \sigma, \sigma \rangle$.
\end{proof}

\begin{lem}
$\langle 2, \kappa, 2\sigma, \sigma \rangle=0$ with zero indeterminacy.
\end{lem}

\begin{proof}
Again, $\langle\kappa, 2 \sigma, \sigma\rangle \subseteq \pi_{29} =0$. And
$$\langle 2, \kappa, 2 \sigma\rangle \supseteq \langle 2, \kappa, 2\rangle\sigma \ni \eta \kappa \sigma =0.$$
The indeterminacy of $\langle 2, \kappa, 2 \sigma\rangle$ is zero. Therefore, this $4$-fold Toda bracket is strictly defined. Again, $\langle 2, \kappa, \pi_{15}\rangle = 0$ for filtration reasons. And the other two parts of the indeterminacy are zero, which follows from the indeterminacy of $\langle 2, \sigma^2, 2 \sigma, \sigma \rangle$.

To see this bracket contains zero, we multiply by $\nu$.
$$\langle 2, \kappa, 2\sigma, \sigma \rangle\nu \subseteq \langle 2, \kappa, \langle 2\sigma, \sigma, \nu \rangle\rangle=\langle 2, \kappa, \nu_4\rangle.$$
Since in the Adams $E_4$-page $\langle h_0, d_0, h_2h_4\rangle = 0$ in the Adams filtration 4, there is an element in $\langle 2, \kappa, \nu_4\rangle$ that is detected by an element in filtration strictly higher than $4$. The indeterminacy of this bracket is $2\pi_{33}+\nu_4\pi_{15}=\nu_4\pi_{15}$, which also contains elements in filtration strictly higher than $4$.  On the other side, $\nu\theta_4$ is detected by $p$ in $Ext^4$. Therefore $\langle 2, \kappa, \nu_4\rangle$ does not contain $\nu\theta_4$. Then the lemma follows from the fact that $\pi_{30}\cong \mathbb{Z}/2$ and is generated by $\theta_4$.
\end{proof}

\begin{rem}
We can show directly that $\langle 2, \kappa, \nu_4\rangle=0$ with zero indeterminacy.
\end{rem}

\section{A few proofs}

We first prove Lemma 2.3 which states that $\sigma\pi_{53}=0$.

\begin{proof}
As shown in \cite{DI1}, $\pi_{53}\cong\mathbb{Z}/2\oplus\mathbb{Z}/2\oplus\mathbb{Z}/2\oplus\mathbb{Z}/2$. One set of generators can be chosen to be elements in $\nu\{h_5c_1\}, \nu\{C\}, \epsilon\{h_3^2h_5\}$ and $\kappa\{u\}$ respectively. Note that $x'$ detects $\epsilon\{h_3^2h_5\}$. Then the lemma follows from $\nu\sigma=0, \epsilon\sigma=0$ and $\kappa\sigma=0$.
\end{proof}











The following lemma is shown by Tangora in \cite{Tan}. We first sketch his proof, then give a more direct proof.

\begin{lem}
$\rho_{15}\theta_4=2\sigma\{h_0^2h_3h_5\}$.
\end{lem}

\begin{proof}
Tangora first showed that $\rho_{15}\theta_4\neq0$ and is detected by $h_0^2h_5d_0$. We have
$$\rho_{15}\theta_4 = \rho_{15} \langle\sigma,2\sigma,\sigma,2\sigma\rangle=\langle\rho_{15} ,\sigma,2\sigma,\sigma\rangle 2\sigma.$$
Then the only possibility is that $\langle\rho_{15} ,\sigma,2\sigma,\sigma\rangle$ is detected by $h_0^2h_3h_5$.\\

We present another proof. In the Adams $E_3$-page, we have $\langle h_3,h_0h_3, h_0^3\rangle = h_0^3h_4$. Therefore, $\rho_{15}$ is contained in $\langle \sigma, 2\sigma, 8\rangle$. Then we have
\begin{equation*}
\begin{split}
\rho_{15}\theta_4 & = \langle \sigma, 2\sigma, 8 \rangle \theta_4 \\
& = \sigma \langle 2\sigma, 8, \theta_4 \rangle \\
& = \sigma \langle 8\sigma, 2, \theta_4 \rangle \\
& = \sigma \{h_0^3h_3h_5\} \\
& = 2 \sigma \{h_0^2h_3h_5\}.
\end{split}
\end{equation*}
For the first equation, $\langle \sigma, 2\sigma, 8 \rangle \theta_4$ has no indeterminacy, hence the equality.
For the last equation, the difference between $\{h_0^3h_3h_5\}$ and $2\{h_0^2h_3h_5\}$ contains elements of higher filtration, namely $\eta\sigma\theta_4$ in this case. The equality holds since $\eta\sigma^2\theta_4=0$.
\end{proof}

Now we prove Lemma 2.4 which states that $\langle\rho_{15}\theta_4, 2\sigma, \sigma\rangle=0$ with zero indeterminacy.

\begin{proof}
The indeterminacy is $\rho_{15}\theta_4\pi_{15}+\sigma \pi_{53}=\rho_{15}\theta_4\pi_{15}$. $\pi_{15}$ is generated by $\eta\kappa$ and $\rho_{15}$. We have $\rho_{15}^2=0$ and $\kappa\theta_4=0$ both for filtration reasons. Therefore the indeterminacy is equal to $\rho_{15}\theta_4\pi_{15}=0$.

By Lemma 6.1, $\langle\rho_{15}\theta_4, 2\sigma, \sigma\rangle=\langle2 \sigma \{h_0^2h_3h_5\}, 2 \sigma, \sigma\rangle$ contains $\sigma \langle 2 \{h_0^2h_3h_5\}, 2 \sigma, \sigma\rangle$. Note that $\langle 2 \{h_0^2h_3h_5\}, 2 \sigma, \sigma\rangle\subseteq\pi_{53}$ and $\sigma\pi_{53}=0$. This completes the proof.
\end{proof}


\begin{thebibliography}{99}

\bibitem{BJM1}
M.G. Barratt, J.D.S. Jones and M.E. Mahowald.
Relations amongst Toda brackets and the Kervaire invariant in dimension $62$.
J. London Math. Soc. 30(1984), 533--550.

\bibitem{BJM2}
M.G. Barratt, J.D.S. Jones and M.E. Mahowald.
The Kervaire invariant problem.
Proceeding of the Northwestern Homotopy Theory Conference (Providence, Rhode Island) (H.R.Miller and S.B.Priddy, eds.) Contemporary Mathematics, vol.19, AMS, 1983, pp 9-22.


\bibitem{BMT}
M.G. Barratt, M.E. Mahowald and M.C.Tangora.
Some differentials in the Adams spectral sequence. II
Topology. 9(1970), 309--316.

\bibitem{Bau}
Tilman Bauer.
Computation of the homotopy of the spectrum tmf.
arXiv:math/0311328

\bibitem{Brd}
W. Browder.
The Kervaire invariant of framed manifolds and its generalization.
Annals of Mathematics 90(1969), 157-186.

\bibitem{Br1}
Robert Bruner.
A new differential in the Adams spectral sequence.
Topology 23(1984), 271-276.

\bibitem{Br2}
Robert Bruner.
The cohomology of the mod 2 Steenrod algebra: a computer calculation. http://www.math.wayne.edu/~rrb/papers/cohom.pdf 

\bibitem{Br3}
Robert Bruner.
Private communication. 2014.

\bibitem{He}
Andre Henriques.
The homotopy groups of tmf and of its localizations.
http://math.mit.edu/conferences/talbot/2007/tmfproc/Chapter16/TmfHomotopy.pdf

\bibitem{HHR}
Michael A. Hill, Michael J. Hopkins and Douglas C. Ravenel.
On the non-existence of elements of Kervaire invariant one.
arXiv:0908.3724

\bibitem{DI1}
Daniel C. Isaksen.
Stable stems.
arXiv:1407.8418. 

\bibitem{DI2}
Daniel C. Isaksen.
Classical and motivic Adams charts.
arXiv:1401.4983.


\bibitem{Ko}
Stanley O. Kochman.
Stable homotopy groups of spheres, a computer-assisted approach.
Lecture Notes in Mathematics 1423, Springer-Verlag, 1990.

\bibitem{Ko2}
Stanley O. Kochman.
Bordism, Stable Homotopy and Adams Spectral Sequences.
Fields Institute Monographs, 7, American Mathematical Society, Fields Institute, 1996.

\bibitem{KM}
Stanley O. Kochman and Mark E. Mahowald.
On the computation of stable stems.
The Cech centennial (Boston, MA, 1993), 299-316, Contemp. Math., 181, Amer. Math. Soc., Providence, RI, 1995.

\bibitem{Lin}
Wen-Hsiung Lin.
A proof of the strong Kervaire invariant in dimension $62$.
First International Congress of Chinese Mathematicians (Beijing, 1998), 351-358,
AMS/IP Stud. Adv. Math., 20, Amer. Math. Soc., Providence, RI, 2001.

\bibitem{MT}
Mark Mahowald and Martin Tangora.
Some differentials in the Adams spectral sequence.
Topology 6 (1967) 349-369.

\bibitem{May}
J.Peter May.
Matric Massey products.
J. Algebra 12(1969), 533-568.

\bibitem{Mil}
R. J. Milgram.
Symmetries and operations in homotopy theory.
Amer. Math. Soc. Proc. Symposia Pure Math., 22(1971), 203-211.

\bibitem{MoT}
R.E. Mosher and M.C. Tangora.
Cohomology operations and applications in homotopy theorey.
Harper and Row, New York, 1968.

\bibitem{Moss}
R.M.F. Moss.
Secondary compositions and the Adams spectral sequence.
Math. Z. 115(1970), 283-310.

\bibitem{PS}
F.P. Peterson and N. Stein.
Secondary cohomology operations: two formulas.
Amer. J. Math., 81(1959), 281-305.

\bibitem{Tan}
Martin Tangora.
Some extension problems in the Adams spectral sequence.
Aarhus Univ., Aarhus, 1970. Mat. Inst.,Aarhus Univ., Aarhus, 1970, pp. 578-587. Various Publ. Ser., No. 13.

\end{thebibliography}
\end{document}